\documentclass[noinfoline,ejs]{imsart}
\setattribute{journal}{name}{}
\setattribute{journal}{issn}{}

\RequirePackage[OT1]{fontenc}
\RequirePackage{amsthm,amsmath}
\RequirePackage[numbers]{natbib}
\RequirePackage[colorlinks,citecolor=blue,urlcolor=blue]{hyperref}
\usepackage{amsfonts}
\usepackage{arydshln} 
\usepackage{multirow}
\usepackage{graphicx} 
\usepackage[usenames]{xcolor}

\startlocaldefs
\numberwithin{equation}{section}
\theoremstyle{plain}

\newtheorem{prop}{Proposition}[section]
\newtheorem{cor}{Corollary}[section]
\newtheorem{lem}{Lemma}[section]
\theoremstyle{remark}

\endlocaldefs

\begin{document}

\begin{frontmatter}

\title{PAC-Bayesian bounds for Principal Component Analysis in Hilbert 
spaces}
\runtitle{PAC-Bayesian bounds for Principal Component Analysis in Hilbert 
spaces}

\begin{aug}
\author{\fnms{Ilaria} \snm{Giulini}\thanksref{t1}\ead[label=e1]{ilaria.giulini@me.com}}

\address{{\sc INRIA} Saclay\\
\printead{e1}}

\thankstext{t1}{The results presented in this paper 
were obtained while the author was preparing her PhD under the 
supervision of Olivier Catoni at the D\'epartement de Math\'ematiques et Applications, \'Ecole Normale Sup\'erieure, Paris, with the financial support of the 
R\'egion \^Ile de France.}

\runauthor{I. Giulini}

\affiliation{D\'epartement de Math\'ematiques et Applications, \'Ecole Normale Sup\'erieure, Paris, France}

\end{aug}

\begin{abstract}
Based on some new robust estimators of the covariance matrix, 
we propose stable versions of Principal Component Analysis (PCA)
and we qualify it 
independently of the dimension of the ambient space.
We first provide a robust estimator of the orthogonal projector on the largest eigenvectors of the covariance matrix. 
The behavior of such an estimator is related to the size of the gap in the spectrum of the covariance matrix and in particular
a large gap is needed in order to get a good approximation. 
To avoid the assumption of a large eigengap in the spectrum of the covariance matrix we propose a robust 
version of PCA that consists in performing a smooth cut-off of the spectrum via a Lipschitz function. 
We provide bounds on the approximation error in terms of the operator norm and of the Frobenius norm. 
\end{abstract}


\tableofcontents
\end{frontmatter}

\section{Introduction}

Principal Component Analysis (PCA) is a classical tool for dimensionality reduction. 
The basic idea of PCA is to reduce the dimensionality of a dataset by projecting it into the space spanned by the directions of maximal variance, 
that are called its principal components.
Since this set of directions lies in the space generated by the eigenvectors associated with the largest eigenvalues of the covariance matrix of the sample,
the dimensionality reduction is achieved by projecting the dataset into the space spanned by these eigenvectors, which in the following we call {\it largest eigenvectors}.

\vskip 2mm
\noindent
Given $X\in \mathbb R^d$ a random vector distributed according to an unknown probability distribution $\mathrm P \in \mathcal M_+^1(\mathbb R^d),$ the goal is to estimate the eigenvalues and eigenvectors of the covariance matrix of $X$
\[
\Sigma = \mathbb E \bigl[ \bigl(X -\mathbb E(X)\bigr)\bigl(X -\mathbb E(X)\bigr)^{\top}  \bigr]
\]
from an i.i.d. sample $X_1, \dots, X_n \in \mathbb R^d$ drawn according to $\mathrm P.$ 
Observe that in the case where the random vector $X$ is centered (i.e. $\mathbb E[X]=0$) the covariance matrix $\Sigma$ is the Gram matrix
\[
G =\mathbb E  \bigl( X X^{\top}  \bigr) .
\]
Many results concerning the Gram matrix estimate can be found in the literature, e.g. ~\cite{Rud}, ~\cite{RVer}, ~\cite{JTr}.
These results follow from the study of random matrix theory and use as an estimator of $G$ 
the matrix obtained by 
replacing the unknown probability distribution $\mathrm P$ with the sample distribution $\frac{1}{n} \sum_{i=1}^n \delta_{X_i}.$
In the following we call such an estimator {\it empirical Gram matrix}.\\[1mm]
However, since the empirical Gram matrix, and consequently classical PCA, 
is sensitive to a heavy tailed sample distribution, 
several methods have been proposed to provide a stabler version of 
PCA, e.g. ~\cite{CLMW}, ~\cite{SMin}. 
In ~\cite{CLMW} the authors show that principal components of a data matrix can be recovered 
when part of the observations are contained in a low-dimensional space and the rest are arbitrarily corrupted.
An alternative approach is proposed in ~\cite{SMin} where, without assuming any geometrical assumption on the data, 
Minsker proposes a robust estimator of the Gram matrix, based on the geometric median. 
Such an estimator is used to provide non-asymptotic dimension-independent results concerning PCA.
\\[1mm]
We use the robust estimator $\hat G$ proposed in ~\cite{Giulini15} to describe a new approach that qualifies the stability of PCA
independently of the dimension of the ambient space. 
Taking advantage of the fact that they are independent of the dimension, the results can be extended to the 
infinite-dimensional setting of separable Hilbert spaces and thus to kernel-PCA. \\[1mm]
Results on PCA in Hilbert spaces can be found in Koltchinskii and Lounici  ~\cite{KoltLou2}, ~\cite{KoltLou3}.
The authors study the problem of estimating the spectral projectors of the covariance operator
by their empirical counterpart in the case of Gaussian centered random vectors, based on the bounds obtained in ~\cite{KoltLou1}, 
and in the setting where both the sample size $n$ and the trace of the covariance operator are large.

\vskip2mm
\noindent
The paper is organized as follows. In section~\ref{prelim} we discuss some preliminary results presented in ~\cite{Giulini15} 
on a robust estimator $\widehat G$ of the Gram matrix. 
In section~\ref{eigensec} we prove that each eigenvalue $\widehat \lambda_i$ of 
$\widehat G$ is a robust estimator of the corresponding eigenvalue of the Gram matrix. 
As a consequence, 
the orthogonal projector on the largest eigenvectors of $G$ can be estimated by the projector on the largest eigenvectors of $\widehat G$, providing a first version of robust PCA,
as shown in section~\ref{spca}.
The behavior of this estimator is related to the size of the gap in the spectrum of the Gram matrix and 
more precisely it is necessary to have a large eigengap in order to get a good approximation
(Proposition~\ref{p51pca}).
To avoid the assumption of a large gap in the spectrum of $G$ 
we propose in section~\ref{rpca} another version of robust PCA 
which consists in performing a smooth cut-off of the spectrum of the Gram matrix via a Lipschitz function. 
We provide bounds on the approximation error, 
in terms of the operator norm (Proposition~\ref{supnorm}) and of the Frobenius norm (Proposition~\ref{Fnorm}), 
that replace the size of the eigengap by the inverse of the Lipschitz constant
of the cut-off function.

\section{Preliminaries}\label{prelim}
In this section we discuss some results presented in ~\cite{Giulini15} concerning the construction of a robust estimator of the Gram matrix. 
The idea is to use some PAC-Bayesian inequalities, linked to Gaussian perturbations, to first construct a confidence region for the quadratic form 
$\theta^{\top} G \theta$ and then to define a robust estimator for such a quantity. 
From a theoretical point of view we can consider any quadratic form belonging to the confidence interval obtained for $\theta^{\top} G \theta$. 
However from an algorithmic point of view, these constraints are imposed only for a finite number of directions.
More precisely, we consider a symmetric matrix $Q$ that satisfies the constraints 
for any $\theta$ in a finite $\delta$-net of the unit sphere $\mathbb S_d = \{ \theta \in \mathbb R^d, \, \| \theta\|=1\}$.
The construction of such an estimator is based on the computation of a convex optimization algorithm. 
\\[1mm]
We first introduce some notation. 
Let, as in the introduction, $X\in \mathbb R^d$ be a random vector of law $\mathrm P\in \mathcal M_+^1(\mathbb R^d)$. 
Let $a>0$ and let 
\[ 
 K  = 1 + \left\lceil a^{-1} \log \biggl( \frac{n}{72(2+c) \kappa^{1/2}} \biggr) \right\rceil
\]
where $c= \frac{15}{8 \log(2)(\sqrt{2}-1)} \exp \left( \frac{1 + 2 \sqrt{2}}{2} \right)$ and
\[
\kappa= \sup_{\substack{
\theta \in \mathbb{R}^d \\ 
\mathbb E ( \langle \theta, X \rangle^2 ) > 0
}} \frac{\mathbb E \bigl( \langle \theta , X \rangle^4 \bigr)}{
\mathbb E \bigl( \langle \theta, X \rangle^2 \bigr)^2}.
\] 
Let $s_4^2 = \mathbb{E} \bigl( \lVert X \rVert^4 \bigr)^{1/2}$ and 
let $\sigma \in ]0, s_4^2]$ be a threshold.
We put
\[
B_*(t) = \begin{cases} 
\displaystyle \frac{n^{-1/2} \zeta(\max \{ t, \sigma \} )}{1 - 4 \, n^{-1/2} \zeta( \max \{ t, \sigma \} )} &  \bigl[ 6 + (\kappa-1)^{-1} \bigr] \zeta( \max \{ t, \sigma \} ) \leq \sqrt{n} \\ 
+ \infty & \text{ otherwise}
\end{cases} 
\]
where 
\begin{multline}\label{zpca}
\zeta (t) = \sqrt{ 2 (\kappa-1) \Biggl( \frac{(2 + 3c) \; s_4^2}{4(2+c) 
\kappa^{1/2} t} 
+ \log(K / \epsilon) \Biggr)} \cosh(a/4)  \\
+ \sqrt{ \frac{2(2+c)\kappa^{1/2} \; 
s_4^2}{ t}} \cosh ( a/ 2).
\end{multline}

\vskip2mm
\noindent
The following proposition holds true.

\begin{prop}
\label{prop1.25} {\bf (~\cite{Giulini15})}
Let us assume that $8 \zeta(\sigma) \leq \sqrt{n}$, $\sigma \leq s_4^2$ 
and that $\kappa 
\geq 3/2$. 
With probability at least $1 - 2 \epsilon$, for any $\theta \in \mathbb{S}_d$, 
\begin{align*}
\Bigl\lvert \max \{ \theta^{\top} Q \theta, \sigma \}  - 
\max \{ \theta^{\top} G \theta, \sigma \}  \Bigr\rvert 
& \leq 2 \max \bigl\{ \theta^{\top} G \theta, \sigma \bigr\} B_* \bigl( 
\theta^{\top} G \theta\bigr) + 5 \delta\lVert G \rVert_{F}, \\
\Bigl\lvert \max \{ \theta^{\top} Q \theta, \sigma \} - 
\max \{ \theta^{\top} G \theta, \sigma \}  \Bigr\rvert 
& \leq 2 \max \bigl\{ \theta^{\top} Q \theta, \sigma \bigr\} B_* \bigl( \min \{ 
\theta^{\top} Q \theta, s_4^2 \bigr\} \bigr) \\
& \hskip 49mm+ 5 \delta\lVert G \rVert_{F}.
\end{align*}
\end{prop}

\vskip2mm
\noindent
We recall that, given $M \in M_d(\mathbb R)$ a symmetric $d\times d$ matrix, the Frobenius norm
of $M$ is defined as
\[
\|M\|_F^2 = \mathbf{Tr}(M^{\top}M)
\]
and that $\| M\|_{\infty} \leq \| M\|_F.$

\vskip2mm
\noindent
Observe that we assume that the threshold $\sigma$ is such that 
$8 \zeta(\sigma) \leq \sqrt n$ in order to have a meaningful bound. 
Indeed, in this case $B_*(t) <+\infty$, provided that $\kappa \geq 3/2$.
\\[1mm]
However, since we do not know whether $Q$ is non-negative, we can decompose it in its positive and negative parts so that 
$Q=Q_+-Q_-$ and consider as an estimator $\widehat G = Q_+$. 
We deduce the following result. 
\begin{prop}
\label{prop1.22q} 
{\bf (~\cite{Giulini15})}
Let us assume that $8 \zeta(\sigma) \leq \sqrt{n}$, 
$\sigma \leq s_4^2$ and that 
$\kappa \geq 3/2$. With probability at least $1 - 2 \epsilon$, 
for any $\theta \in \mathbb{S}_d$, 
\begin{align*}
\Bigl\lvert \max \{ \theta^{\top} \widehat G \theta, \sigma \}  - 
\max \{ \theta^{\top} G \theta, \sigma \}  \Bigr\rvert 
& \leq 2 \max \bigl\{ \theta^{\top} G \theta, \sigma \bigr\} B_* \bigl( 
\theta^{\top} G \theta\bigr) + 7 \delta\lVert G \rVert_{F}, \\
\Bigl\lvert \max \{ \theta^{\top} \widehat G \theta, \sigma \} - 
\max \{ \theta^{\top} G \theta, \sigma \}  \Bigr\rvert 
& \leq 2 \max \bigl\{ \theta^{\top}\widehat G \theta, \sigma \bigr\} B_* \bigl( \min \{ 
\theta^{\top} \widehat G \theta, s_4^2 \bigr\} \bigr) \\
& \hskip 49mm+ 7 \delta \lVert G \rVert_{F}.
\end{align*}
\end{prop}

\section{Estimate of the eigenvalues} \label{eigensec}

Denote by $\lambda_1\geq \dots \geq \lambda_d$ the eigenvalues of $G$ and 
by $p_1, \dots, p_d$ a corresponding orthonormal basis eigenvectors, 
so that $\lambda_i= p_i^{\top}Gp_i.$ \\[1mm]
Similarly, let $\widehat \lambda_1\geq \dots \geq \widehat \lambda_d$ be the 
eigenvalues of $\widehat G$ and $q_1, \dots, q_d$ a corresponding orthonormal basis of eigenvectors.

In this section we prove that each eigenvalue of $\widehat G$ 
is a robust estimator of the corresponding eigenvalue of the Gram matrix.

\begin{prop}
\label{prop1.23eig} 
Let us assume that $8 \zeta(\sigma) \leq \sqrt{n}$, $\sigma \leq 
s_4^2$ and that $\kappa \geq 3/2$. 
With probability at least $1 - 2 \epsilon$, 
for any $i = 1, \dots, d$, the two following inequalities 
hold together
\begin{align*}
\bigl\lvert \max\{ \lambda_i , \sigma\}- \max\{ \widehat{\lambda}_i , \sigma\}\bigr\rvert & \leq 
2 \max\{ \lambda_i , \sigma\}B_*(\lambda_i) + 5 \delta \lVert G \rVert_{F}, \\ 
\bigl\lvert\max\{ \lambda_i, \sigma\}-\max\{  \widehat{\lambda}_i, \sigma\} \bigr\rvert & \leq 
2 \max\{ \widehat{\lambda}_i , \sigma\} B_* \bigl( \min \bigl\{ 
\widehat{\lambda}_i, s_4^2 \bigr\} \bigr) + 5 \delta \lVert G \rVert_{F}. 
\end{align*}
Consequently, 
\begin{align*}
|\lambda_i - \widehat \lambda_i | & \leq 2 \max\{ \lambda_i , \sigma\}
B_* \bigl( \lambda_i \bigr) + 5 \delta \lVert G \rVert_{F} + \sigma,\\
|\lambda_i - \widehat \lambda_i  | & \leq 2 \max\{\widehat \lambda_i , \sigma\}  
B_* \bigl( \min \bigl\{ \widehat \lambda_i, s_4^2 \bigr\} \bigr)  + 5 \delta 
\lVert G \rVert_{F} + \sigma.
\end{align*}
\end{prop}

\begin{proof}
We observe that, for any $i \in \{ 1,\dots, d\},$ the vector space 
\[
\mathbf{span}\{q_1,\dots, q_{i-1} \}^{\perp} \cap \mathbf{span} 
\{p_1,\dots, p_i\} \subset \mathbb R^d
\]
is of dimension at least 1, so that the set
\[
V_i = \bigl\{ \theta \in \mathbb S_d \ | \ \theta \in 
\mathbf{span} \{q_1,\dots, q_{i-1} \}^{\perp} \cap 
\mathbf{span} \{p_1,\dots, p_i\} \bigr\} \subset \mathbb R^d
\]
is non-empty. 

Indeed, putting $A  = \mathbf{span} \{ q_1, \dots, q_{i-1} \}^{\perp}$ 
and $B = \mathbf{span} \{p_1, \dots, p_i\}$, 
we see that $\dim(A \cap B) = \dim(A) + \dim(B) - \dim(A+B) \geq 1$, 
since $\dim(A+B) \leq \dim(\mathbb{R}^d) = d$ and $\dim(A) + \dim(B) = d+1$. 
Hence, there exists $\theta_i \in V_i$ 
and for such a $\theta_i$, we have $\theta_i^{\top}G\theta_i\geq \lambda_i.$ It follows that
\begin{align*}
\max\{ \lambda_i , \sigma\} & \leq \sup \left\{ \max\{\theta^{\top}G\theta, \sigma\}   \ | \  \theta \in V_i \right\}  \\
& \leq \sup \Bigl\{ \max\{\theta^{\top}G\theta , \sigma\}  \ | \  \theta \in \mathbb{S}_d,\ \theta \in \mathbf{span} \{q_1, \dots, q_{i-1} \}^{\perp} \Bigr\} .
\end{align*}
Therefore, according to Proposition \ref{prop1.25}, 
\begin{multline*}
\max \{ \lambda_i, \sigma \} \bigl( 1 - 2 B_*(\lambda_i) \bigr) \\ 
\leq \sup \Bigl\{ \max \{ \theta^{\top} Q \theta, \sigma \} \; | \;  
\theta \in \mathbb{S}_d \cap \mathbf{span} \{ q_1, \dots, q_{i-1} \}^{\perp} \Bigr\}+ 5 \delta \lVert G \rVert_{F} \\ 
 \leq \max \{ \widehat{\lambda}_i , \sigma \} + 5 \delta \lVert G \rVert_{F}. 
\end{multline*}
In the same way,
\begin{align*}
\max \{ \widehat{\lambda}_i, \sigma \} & \leq \sup \Bigl\{ 
\max \{ \theta^{\top}G\theta, \sigma \}  \Bigl( 1 + 2 B_*\bigl( \theta^{\top}G\theta \bigr) \Bigr) \; 
\\ & \hspace{30ex} | \; \theta \in \mathbb{S}_d \cap \mathbf{span} \{ p_1, \dots, p_{i-1}\}^{\perp} \Bigr\} 
\\ & \hspace{3ex} + 5 \delta \lVert G \rVert_{F}\\ 
& \leq \max \{ \lambda_i, \sigma \} \bigl( 1  + 2 B_*(\lambda_i) \bigr) 
+ 5 \delta \lVert G \rVert_{F}, 
\end{align*}
\begin{multline*} 
\max \{ \widehat{\lambda}_i, \sigma \} \Bigl( 1 - 2 B_* \bigl( \min \{ 
\widehat{\lambda}_i, s_4^2 \} \bigr) \Bigr)\\
\leq \sup \Bigl\{ \max \{ \theta^{\top}G\theta , \sigma \} \; | \; \theta \in \mathbb{S}_d \cap 
\mathbf{span} \{ p_1, \dots, p_{i-1} \}^{\perp} \Bigr\} + 5 \delta \lVert G \rVert_{F} \\
\leq \max \{ \lambda_i, \sigma \} + 5 \delta \lVert G \rVert_{F},
\end{multline*}
and
\begin{align*}
\max \{ \lambda_i, \sigma \} &
\leq \sup \Bigl\{ \max \{ \theta^{\top} Q \theta, \sigma \} \Bigl( 1 + 2 B_*\bigl( \min \{ \theta^{\top} Q \theta, s_4^2 \} \bigr) \Bigr) \; | \\
& \hskip21mm | \; \theta \in \mathbb{S}_d \cap \mathbf{span} \{ q_1, \dots, q_{i-1} \}^{\perp} \Bigr\} 
+ 5 \delta \lVert G \rVert_{F} \\ 
& \leq \max \{ \widehat{\lambda}_i, \sigma \} \Bigl( 1 + 2 B_* \bigl( \min \{ \widehat{\lambda}_i, s_4^2 \} \bigr) \Bigr) + 5 \delta \lVert 
G \rVert_{F}.
\end{align*}  
In all these inequalities we have used the fact that 
\begin{align*}
t & \mapsto \max \{t, \sigma \}  \Bigl( 1 - 2 B_*(\min \{ t, s_4^2 \} ) \Bigr) \\
t & \mapsto \max \{t, \sigma \}  \Bigl( 1 + 2 B_*(\min \{ t, s_4^2 \} ) \Bigr)
\end{align*}
are non-decreasing and that $\lambda_i \leq s_4^2$.\\
This proves the proposition for the eigenvalues of $Q$, and therefore 
also for their positive parts, that are the eigenvalues of $\widehat G=Q_+$. \\[1mm]
To prove the second part of the proposition, it is sufficient to observe that 
\[
|\lambda_i - \widehat \lambda_i |  \leq \bigl\lvert \max\{ \lambda_i , 
\sigma\}- \max\{ \widehat{\lambda}_i , \sigma\}\bigr\rvert + \sigma.
\]
\end{proof}

\vskip2mm
\noindent
Given a threshold $\sigma\leq s_4^2$ such that $8\zeta(\sigma) \leq \sqrt n,$ 
since the bound
\[
F(t)= \max\{ t, \sigma\} B_*(\min\{ t, s_4^2\})
\]
obtained in Proposition~\ref{prop1.23eig} is non-decreasing for any $t \in \mathbb R_+$, 
we get the following result: 

\begin{cor}\label{cor.eigen} Under the same assumptions as in Proposition ~\ref{prop1.23eig}, with probability at least $1-2\epsilon,$ for any $i =1,\dots, d,$
\begin{align*}
|\lambda_i - \widehat \lambda_i | & \leq 2 \max\{ \lambda_1 , \sigma\}
B_* \bigl( \lambda_1 \bigr) + 5 \delta \lVert G \rVert_{F} + \sigma,\\
|\lambda_i - \widehat \lambda_i  | & \leq 2 \max\{\widehat \lambda_1 , \sigma\}  
B_* \bigl( \min \bigl\{ \widehat \lambda_1, s_4^2 \bigr\} \bigr)  + 5 \delta 
\lVert G \rVert_{F} + \sigma.
\end{align*}
\end{cor}

\vskip 2mm
\noindent
In order to simplify notation, we define
\begin{equation}\label{eq.defB}
B(t) = 2 \max\{ t, \sigma\} B_*\left( \min\{ t, s_4^2\} \right) + 7 \delta 
\lVert G \rVert_{F} + \sigma.
\end{equation}
Remark that, since $B_*$ is non-increasing, $F$ is non-decreasing
and $B_*(t) \leq 1/4$, for any $a \in \mathbb{R}_+$, we have 
$B(t + a) \leq B(t)  + a/2$.

\section{Robust PCA}\label{spca}

A method to determine the number of relevant components is based on the difference in magnitude between successive eigenvalues. 
In this section we study the projection on the 
$r$ largest eigenvectors $p_{1}, \dots, p_r$ of the Gram matrix, 
assuming that there is a gap in the spectrum of the Gram matrix, 
meaning that 
\begin{equation}\label{defr}
\lambda_{r}- \lambda_{r+1}>0.
\end{equation}
We denote by $\Pi_r$ the orthogonal projector on the $r$ largest eigenvectors $p_{1}, \dots, p_r$ of $G$ and similarly by $\widehat \Pi_r$ the orthogonal projector on the $r$ largest eigenvectors $q_{1}, \dots, q_r$ of 
its estimate $\widehat G$. \\[1mm]
Our goal is to provide a bound on the approximation $\|\Pi_r-\widehat \Pi_r\|_{\infty}$ that does not depend explicitly on the dimension $d$ of the ambient space.

\begin{prop}\label{p51pca}  
With probability at least $1-2\epsilon,$
\[
\|\Pi_r-\widehat \Pi_r\|_{\infty} \leq  \frac{\sqrt{2r}}{\lambda_{r}-\lambda_{r+1} }   B\left(\lambda_1 \right)
\]
where $B$ is defined in equation ~\eqref{eq.defB} and $\lambda_1$ is the largest eigenvalue of the Gram matrix. 
\end{prop}

\vskip1mm
\noindent
For the proof we refer to section~\ref{pf1}.

\vskip 2mm
\noindent
We observe that the above proposition provides a bound on the approximation error $\|\Pi_r-\widehat \Pi_r\|_{\infty}$ 
that does not depend explicitly on the dimension $d$ of the ambient space, since $B$ 
is dimension-free. 
However the result relates the quality of the approximation of the orthogonal projector $\Pi_r$ 
by the robust estimator $\widehat \Pi_r$ to the size of the spectral gap and 
in particular the larger the eigengap, the better the approximation is. 
A way to estimate the size of the eigengap is by using the eigenvalues of the robust estimator $\widehat G$, 
since, according to Proposition~\ref{prop1.23eig}, each eigenvalue of $\hat G$ provides a good approximation
of the corresponding eigenvalue of the Gram matrix.

\section{Robust PCA with a smooth cut-off}\label{rpca}

In order to avoid the requirement of a large spectral gap, we can 
interpret the projector $\Pi_r$ as a step function applied 
to the spectrum of the Gram matrix and consider to replace 
$\Pi_r$ with a smooth cut-off of the eigenvalues of $G$ via a 
Lipschitz function.
More specifically, we have in mind to apply to 
the spectrum of $G$ a Lipschitz function that takes the value one 
on the largest eigenvalues and the value zero on the smallest ones.

\vskip 2mm
\noindent 
Let $f$ be a Lipschitz function with Lipschitz constant ${1}/{L}.$  
We decompose the Gram matrix as 
\[
G= UDU^{\top}
\]
where $D= \mathrm{diag}(\lambda_1, \dots, \lambda_d)\in M_d(\mathbb R)$ is the diagonal matrix whose entries are the eigenvalues of $G$ and $U\in M_d(\mathbb R)$ is the orthogonal matrix of eigenvectors of $G.$ 
We define $f(G)$ as 
\[
f(G)=  U \mathrm{diag} \Bigl( f(\lambda_1), \dots, f(\lambda_d) \Bigr) U^{\top}
\]
where, for any $i,j \in \{1,\dots, d\},$
\[
|f(\lambda_i) - f(\lambda_j)| \leq \frac{1}{L} | \lambda_i - \lambda_j|.
\]

\vskip 2mm
\noindent
We provide some results on the estimate of $f(G)$, the image of the Gram 
matrix by the smooth cut-off $f$, in terms of the operator norm $\|\cdot\|_{\infty}$ and of the Frobenius norm $\|\cdot\|_F.$ 
We start with a general result.

\begin{prop}\label{fbn} Let $M,$ $M' \in M_d(\mathbb R)$ be two symmetric matrices. We denote by $\mu_1, \dots, \mu_d$ the eigenvalues of $M$ related to the orthonormal basis of eigenvectors $p_1,\dots, p_d$ and by $\mu'_1, \dots, \mu'_d$ the eigenvalues of $M'$ related to the orthonormal basis of eigenvectors $q_1,\dots, q_d.$ We have 
\begin{equation}\label{eq_fbn} \| M-M'\|_F^2 = \sum_{i,k=1}^d (  \mu_i -  \mu'_k)^2 \langle p_i, q_k\rangle^2.
\end{equation}
Moreover, let $f$ be a ${1}/{L}$-Lipschitz function. We have 
\begin{equation}\label{rcorf} \| f(M)-f(M')\|_F \leq \frac{1}{L} \| M-M'\|_F .
\end{equation}
\end{prop}

\vskip2mm
\noindent
For the proof we refer to section ~\ref{pf2}.\\[1mm]
Let us now present the bound on the approximation error $\| f(G)- f(\widehat G) \|_{\infty}$ in terms of the
operator norm.

\begin{prop}\label{supnorm}{\bf (Operator norm)}
With probability at least $1-2\epsilon$, 
for any $1/L$-Lipschitz function $f$, 
\begin{align*}
\| f(G)- f(\widehat G) \|_{\infty}& \leq \min_{r \in \{1, \dots, 
d \}} L^{-1}\left( B\left(\lambda_1 \right)+ \sqrt{4r B\left(\lambda_1 \right)^2 + 2\sum_{i=r+1}^{d} \lambda_i^2} \; \right),
\end{align*}
 where $B$ is defined in equation ~\eqref{eq.defB} and $\lambda_1 \geq \dots \geq \lambda_d$ are the eigenvalues of $G$.
\end{prop}

\vskip2mm
\noindent
For the proof we refer to section~\ref{pf3}.\\[1mm]
Observe that with respect to the bound obtained in Proposition~\ref{p51pca}, we have replaced the inverse of the size of the gap
with the Lipschitz constant. 
Moreover, in the case it exists a gap $\lambda_r- \lambda_{r+1}>0$, there exists a Lipschitz function $f$ 
such that $\Pi_r = f(G)$ and whose Lipschitz constant is exactly the inverse of the size of the gap. 
Otherwise, if we want to use $f$ with a better Lipschitz constant, we have to approximate $\Pi_r$ with the smoother
approximate projection $f(G)$. 

\vskip 3mm
\noindent
Slightly changing the definition of the estimator we can obtain a bound for the approximation error in terms of the Frobenius norm. 
Instead of considering $\widehat G = \sum_{i=1}^d \widehat \lambda_i q_i q_i^{\top}$ we consider the matrix
$$\widetilde G = \sum_{i=1}^d \widetilde \lambda_i q_i q_i^{\top}$$
with eigenvectors $q_1,\dots, q_d$ and eigenvalues 
\[
\widetilde \lambda_i = \left[\widehat \lambda_i -B(\widehat \lambda_i) \right]_+
\]
where we recall that $B$ is defined in equation ~\eqref{eq.defB}. 
We observe that, in the event of probability at least $1-2\epsilon$ described in Corollary~\ref{cor.eigen}, for any $i = 1,\dots, d,$
\[\widetilde \lambda_i \leq \lambda_i.\]

\vskip2mm
\noindent
According to Proposition~\ref{fbn}, we first present a result on the approximation error $\| G - \widetilde G\|_F$.

\begin{prop}\label{Fnorm} {\bf (Frobenius norm)}
With probability at least $1-2\epsilon,$
\[
\| G - \widetilde G\|_F \leq 
\min_{r \in \{1 \dots, d\}} \sqrt{13 r B(\lambda_1)^2+ 2 \sum_{i=r+1}^{d}\lambda_i^2},
\]
 where $B$ is defined in equation ~\eqref{eq.defB} and $\lambda_1 \geq \dots \geq \lambda_d$ are the eigenvalues of $G$.
\end{prop}

\vskip2mm
\noindent
For the proof we refer to section~\ref{pf4}.

\vskip 2mm
\noindent
To obtain a bound on $\| f(G) - f(\widetilde G) \|_F$ it is sufficient to combine the above proposition with Proposition~\ref{rcorf}. 

\begin{cor}\label{cor1.30pca} With the same notation as in Proposition~\ref{Fnorm}, 
with probability at least $1-2\epsilon$, for any $1/L$-Lipschitz 
function $f$, 
\[
\| f(G) - f(\widetilde G) \|_F \leq \min_{r \in \{1, \dots, d\}} L^{-1} \sqrt{13 \, r B(\lambda_1)^2+ 2 \sum_{i=r+1}^{d}\lambda_i^2}.
\]
\end{cor}

\vskip 2mm
\noindent
In the previous bounds, the optimal choice of the dimension parameter 
$r$ depends on the distribution of the eigenvalues of the 
Gram matrix $G$. Nevertheless, it is possible to upper 
bound what happens when this distribution of eigenvalues 
is the worst possible.\\[1mm]
Observe that
\[
 \sum_{i=r+1}^{d}\lambda_i^2 \leq \lambda_{r+1}\mathbf{Tr}(G)
 \]
 and also $r\lambda_{r+1} \leq \mathbf{Tr}(G),$ so that 
  \[
 \sum_{i=r+1}^{d}\lambda_i^2 \leq r^{-1}\mathbf{Tr}(G)^2.
 \]
Hence, if we consider for example the case where the approximation error is evaluated in terms of the Frobenius norm, the worst case formulation of Corollary ~\ref{cor1.30pca} is obtained choosing 
\[
 r= \Bigl\lceil \sqrt{2/13} \mathbf{Tr}(G) B(\lambda_1)^{-1} \Bigr\rceil
\]
and, in this case, it can be restated as follows. 
 
 \begin{cor}
With probability at least $1-2\epsilon,$
\[
\| f(G) - f(\widetilde G) \|_F \leq L^{-1} \sqrt{11 \, \mathbf{Tr}(G) B(\lambda_1)+13 \, B(\lambda_1)^2}.
\]
\end{cor}
\noindent
This proposition shows that the worst case speed is not slower than $n^{-1/4}$.
We do not know whether this rate is optimal in the worst case. 
We could in the same way obtain a worst case corollary for Proposition~\ref{supnorm}. 

\section{Proofs}
In this section we give the proofs of the results presented in the previous sections. 
More precisely, section ~\ref{pf1} refers to Proposition~\ref{p51pca},
section ~\ref{pf2} refers to Proposition~\ref{fbn}, 
section ~\ref{pf3} refers to Proposition~\ref{supnorm} and 
section ~\ref{pf4} refers to Proposition~\ref{Fnorm}.
We start with a technical lemma that will be useful in several proofs.

\begin{lem}\label{lem1.29} With probability at least $1-2\epsilon,$ for any $k\in \{1,\dots, d\},$ the two inequalities hold together
\begin{align}
\label{lem1.29eq1}
 \sum_{i=1}^d \left( \lambda_i-\lambda_k \right)^2 \langle q_k, p_i \rangle^2 & \leq 2B\left(\lambda_1 \right)^2 \\
  \label{lem1.29eq2}
 \sum_{i=1}^d \left( \lambda_i-\widehat \lambda_k \right)^2 \langle q_k, p_i \rangle^2 & \leq B\left(\lambda_1 \right)^2,
 \end{align}
where $B(t)= 2 \max\{ t, \sigma\} B_*\left( \min\{ t, s_4^2\} \right) + 7 \delta 
\lVert G \rVert_{F} + \sigma$ is defined in equation ~\eqref{eq.defB}.
\end{lem}

\begin{proof} We observe that, 
\begin{align*}
\lVert G - \widehat{G} \rVert_{\infty} & = \max \Bigl\{ 
\sup_{\theta \in \mathbb{S}_d} 
\theta^{\top} \bigl( G - \widehat{G} \bigr) \theta , \ 
\sup_{\theta \in \mathbb{S}_d} 
\theta^{\top} \bigl( \widehat{G} - G \bigr) \theta 
\Bigr\}  
\\ & = \sup_{\theta \in \mathbb{S}_d} \; \bigl\lvert \theta^{\top} G \theta - \theta^{\top} \widehat G \theta 
\bigr\rvert\\
& \leq \sup_{\theta \in \mathbb{S}_d} \; \bigl\lvert \max\{ \theta^{\top} G \theta, \sigma\} - \max\{ \theta^{\top} \widehat G \theta , \sigma\}
\bigr\rvert + \sigma
\end{align*}
for any threshold $\sigma>0$. 
Thus, by Proposition ~\ref{prop1.22q}, with probability at least $1-2\epsilon,$
\begin{align*} 
\sup_{\theta \in \mathbb{S}_d} \| G\theta - \widehat G \theta \|  \leq  B\left(\lambda_1 \right).
\end{align*}
To prove equation ~\eqref{lem1.29eq2} it is sufficient to observe that, since 
\[
\| G\theta - \widehat G \theta \| = \| \sum_{i,j=1}^d(\lambda_i-\widehat \lambda_j) \langle \theta, q_j \rangle\langle p_i, q_j \rangle p_i\|
\] 
choosing $\theta = q_k,$ with $k \in \{1,\dots, d\}$,
 \[
\| Gq_k - \widehat G q_k \|^2 = \sum_{i=1}^d(\lambda_i-\widehat \lambda_k)^2 \langle  q_k, p_i \rangle^2 .
\] 

\noindent
On the other hand, to prove equation ~\eqref{lem1.29eq1}, we observe that 
$$\aligned\| G\theta - \widehat G \theta \| & = \| \sum_{i=1}^d  \lambda_i \langle \theta, p_i\rangle p_i -  \sum_{i=1}^d \widehat   \lambda_i \langle \theta, q_i\rangle q_i \|\\
& = \| \sum_{i=1}^d   \lambda_i \left(\langle \theta, p_i\rangle p_i -\langle \theta, q_i\rangle q_i \right)-  \sum_{i=1}^d \left( \widehat   \lambda_i -   \lambda_i \right)\langle \theta, q_i\rangle q_i \|\\
& \geq  \| \sum_{i=1}^d   \lambda_i\left(\langle \theta, p_i\rangle p_i -\langle \theta, q_i\rangle q_i \right)\|- \| \sum_{i=1}^d \left( \widehat   \lambda_i -   \lambda_i \right)\langle \theta, q_i\rangle q_i \|
\endaligned$$

\noindent
where, by Corollary ~\ref{cor.eigen},

\begin{align*} 
\| \sum_{i=1}^d \left( \widehat   \lambda_i -  \lambda_i \right)\langle \theta, q_i\rangle q_i \|^2 & = \sum_{i=1}^d \left( \widehat   \lambda_i -   \lambda_i \right)^2\langle \theta, q_i\rangle^2 \leq B\left(\lambda_1 \right)^2.
\end{align*}

\noindent
Choosing again $\theta=  q_k,$ for $k \in \{1,\dots, d\},$ we get

\begin{align*}
\| \sum_{i=1}^d  \lambda_i\left(\langle q_k, p_i\rangle p_i -\langle q_k, q_i\rangle q_i \right)\|^2 & = \| \sum_{i,j=1}^d \left( \lambda_i-\lambda_j \right) \langle q_k, q_j\rangle \langle q_j, p_i \rangle  p_i\|^2\\
& =  \| \sum_{i=1}^d \left( \lambda_i-\lambda_k \right) \langle q_k, p_i \rangle p_i\|^2\\
& =  \sum_{i=1}^d \left( \lambda_i-\lambda_k \right)^2 \langle q_k, p_i \rangle^2,
\end{align*}
which concludes the proof.
\end{proof}

\subsection{Proof of Proposition~\ref{p51pca}}\label{pf1}

Since $\Pi_r$ and $\widehat \Pi_r$ have the same rank, we can write

\[
\|\Pi_r-\widehat \Pi_r\|_{\infty}=\sup_{\substack{\theta \in \mathbb{S}_d \\ \theta \in \mathbf{Im}(\widehat \Pi_r)}} \| \Pi_r\theta-\widehat \Pi_r\theta\|
\]
as shown in Lemma ~\ref{imQ} in Appendix ~\ref{appx}.
Moreover, for any $\theta\in \mathbf{Im}(\widehat \Pi_r)\cap \mathbb{S}_d,$  we observe that
$$\aligned  \| \Pi_r\theta-\widehat \Pi_r\theta\|^2 & = \| \Pi_r\theta-\theta\|^2 \\
& =\| \sum_{i=1}^r \langle \theta, p_i\rangle p_i - \sum_{i=1}^d \langle \theta, p_i\rangle p_i \|^2\\
& = \sum_{i=r+1}^d \langle \theta, p_i\rangle^2.
\endaligned$$

\noindent
Since any $\theta \in \mathbf{Im}(\widehat \Pi_r)$ can be written as $\theta= \sum_{k=1}^r \langle \theta, q_k\rangle q_k$ with $\sum_{k=1}^r \langle \theta, q_k\rangle^2=1,$ then
$$\aligned  \| \Pi_r\theta-\widehat \Pi_r\theta\|^2 &= \sum_{i=r+1}^d \left( \sum_{k=1}^r \langle \theta, q_k \rangle \langle q_k,p_i\rangle \right)^2.
\endaligned$$

\noindent
Hence, by the Cauchy-Schwarz inequality, we get
\begin{align} 
\| \Pi_r\theta-\widehat \Pi_r\theta\|^2& \leq  \sum_{i=r+1}^d \left( \sum_{k=1}^r \langle \theta, q_k \rangle^2\right)  \left( \sum_{k=1}^r \langle q_k,p_i\rangle^2\right)  \\
 \label{PtQt}
& = \sum_{k=1}^r \sum_{i=r+1}^d \langle q_k,p_i\rangle^2.
\end{align}

\vskip2mm
\noindent
Moreover, for any $k \in  \{1,\dots, r\},$ we have
$$\aligned  
\sum_{i=r+1}^d \left(  \lambda_{k}-\lambda_{i} \right)^2 \langle q_k, p_i \rangle^2 & \geq  \sum_{i=r+1}^d \left( \lambda_{r}-\lambda_{i} \right)^2 \langle q_k, p_i \rangle^2\\
& \geq \left( \lambda_{r}-\lambda_{r+1} \right)^2 \sum_{i=r+1}^d  \langle q_k, p_i \rangle^2 .
\endaligned$$

\noindent
Then, by Lemma~\ref{lem1.29}, with probability at least $1-2\epsilon,$
$$\left(\lambda_{r}-\lambda_{r+1} \right)^2 \sum_{i=r+1}^d  \langle q_k, p_i \rangle^2 \leq 2B\left(\lambda_1 \right)^2.$$

\noindent
Applying the above inequality to equation ~\eqref{PtQt}
we conclude the proof.

\subsection{Proof of Proposition~\ref{fbn}}\label{pf2}

Since $\{p_i\}_{i=1}^d$ is an orthonormal basis of eigenvectors of $M$ and $\{\mu_i\}_{i=1}^d$ the corresponding eigenvalues, we can write $M$ as
$$M= \sum_{i=1}^d \mu_i p_i p_i^{\top}.$$ 
Similarly, $$M' = \sum_{i=1}^d \mu'_i q_i q_i^{\top}.$$
Our goal is to evaluate $\| M - M'\|_F$ where, by definition,
$$\aligned  M- M' & =  \sum_{i=1}^d \mu_i p_i p_i^{\top} - \sum_{k=1}^d \mu'_k q_k q_k^{\top}\\
& = \sum_{i,k=1}^d (  \mu_i -  \mu'_k) \langle p_i, q_k\rangle q_k p_i^{\top}.
\endaligned$$

\noindent
We observe that $M-M'$ is a symmetric matrix and its Frobenius norm is
$$\| M-M'\|_F^2 = \mathbf{Tr}((M-M')^{\top} (M-M')),$$

\noindent
where
$$\aligned (M-M')^{\top} (M-M')& = \sum_{i,j,k=1}^d (  \mu_i -  \mu'_k)(  \mu_i -  \mu'_j) \langle p_i, q_k\rangle  \langle p_i, q_j\rangle q_k q_j^{\top}	.
\endaligned$$

\noindent
Considering that $\mathbf{Tr}(q_k q_j^{\top})= \delta_{jk},$ we conclude that
$$\aligned \| M-M'\|_F^2 & =\sum_{i,j,k=1}^d (  \mu_i -  \mu'_k)(  \mu_i -  \mu'_j) \langle p_i, q_k\rangle  \langle p_i, q_j\rangle\delta_{jk}\\
& = \sum_{i,k=1}^d (  \mu_i -  \mu'_k)^2 \langle p_i, q_k\rangle^2.
\endaligned$$

\vskip1mm
\noindent
To prove the second part of the result, it is sufficient to observe that using twice equation~\eqref{eq_fbn}. Indeed 
\begin{align*}
\| f(M)-f(M')\|_F^2 & = \sum_{i,k=1}^d ( f( \mu_i) - f( \mu'_k))^2 \langle p_i, q_k\rangle^2\\
& \leq  \frac{1}{L^2}  \sum_{i,k=1}^d (  \mu_i -  \mu'_k)^2 \langle p_i, q_k\rangle^2 = \frac{1}{L^2}\| M-M'\|_F^2. 
\end{align*}

\subsection{Proof of Proposition~\ref{supnorm}}\label{pf3}

In all this proof, we will assume that the event of probability at least $1-2\epsilon$ described in Proposition~\ref{prop1.23eig} holds true. Let  $H \in M_d(\mathbb R)$ be the matrix defined as
\[
H= \sum_{k=1}^d \lambda_k q_k q_k^{\top}.
\]
We observe that 
\[
\| f(G)- f(\widehat G) \|_{\infty} \leq \| f(G)- f(H) \|_{\infty} + \| f(H)- f(\widehat G) \|_{\infty}
\]

\noindent
and we look separately at the two terms. By definition of operator norm, 
we have 
$$ \aligned\| f(H)- f(\widehat G) \|_{\infty}^2 & = \sup_{\theta\in \mathbb S_d} \| f(H)\theta- f(\widehat G)\theta \|^2\\
& =  \sup_{\theta\in \mathbb S_d} \| \sum_{k=1}^d (f(\lambda_k) -f(\widehat \lambda_k)) \langle \theta, q_k \rangle q_k \|^2\\
& =  \sup_{\theta\in \mathbb S_d} \sum_{k=1}^d (f(\lambda_k) -f(\widehat \lambda_k))^2 \langle \theta, q_k \rangle^2.
\endaligned$$ 

\noindent
Since the function $f$ is $1/L$-Lipschitz, we get
$$ \| f(H)- f(\widehat G) \|_{\infty}^2  \leq L^{-2}  \sup_{\theta\in \mathbb S_d} \sum_{k=1}^d (\lambda_k -\widehat \lambda_k)^2 \langle \theta, q_k \rangle^2$$
\noindent
and then, applying Corollary ~\ref{cor.eigen}, with  probability at least $1-2\epsilon,$ we obtain 
\[
\| f(H)- f(\widehat G) \|_{\infty}^2  \leq L^{-2} B\left(\lambda_1 \right)^2.
\] 

\noindent
On the other hand, we have 
$$\aligned  \| f(G)- f(H) \|_{\infty} &  \leq  \| f(G)- f(H) \|_F\\
& \leq \frac{1}{L} \|G-H\|_F,
\endaligned$$
as shown in equation ~\eqref{rcorf}. Hence, according to Proposition~\ref{fbn}, we get
\[
\| f(G)- f(H) \|_{\infty}^2 \leq  \frac{1}{L^2} \sum_{i,k=1}^d (\lambda_i-\lambda_k)^2 \langle p_i, q_k\rangle^2
\]
where, for any $r,$
\begin{align*}&\sum_{i,k=1}^d (\lambda_i-\lambda_k)^2 \langle p_i, q_k\rangle^2 \leq \left( \sum_{i=1}^r \sum_{k=1}^d+ \sum_{i=1}^d \sum_{k=1}^r+ \sum_{i,k=r+1}^{d}\right) (\lambda_i-\lambda_k)^2 \langle p_i, q_k\rangle^2.
\end{align*}

\noindent
Since $\lambda_i \geq 0,$ for any $i\in \{1,\dots, d\},$ we get
$$\aligned \sum_{i,k=r+1}^{d} (\lambda_i-\lambda_k)^2 \langle p_i, q_k\rangle^2 
& \leq 2\sum_{i=r+1}^{d} \lambda_i^2.
\endaligned$$

\noindent
Moreover, by Lemma~\ref{lem1.29}, we have 
\[
 \sum_{k=1}^r\sum_{i=1}^d  (\lambda_i-\lambda_k)^2 \langle p_i, q_k\rangle^2
 \leq 2 r  B\left(\lambda_1 \right)^2
\]

\noindent
and since the same bound also holds for
\[
\sum_{i=1}^r \sum_{k=1}^d (\lambda_i-\lambda_k)^2 \langle p_i, q_k\rangle^2, 
\]
we conclude the proof.

\subsection{Proof of Proposition~\ref{Fnorm}}\label{pf4}

During the whole proof, we will 
assume that the event of probability at least $1-2\epsilon$ 
described in Proposition~\ref{prop1.23eig} is realized. 
According to  Proposition~\ref{fbn}, we have 
\[ 
\| G - \widetilde G\|_F^2 = \sum_{i,k=1}^d (  \lambda_i -  \widetilde \lambda_k)^2 \langle p_i, q_k\rangle^2,
\]
where
\[
\sum_{i,k=1}^d (\lambda_i - \widetilde \lambda_k)^2 \langle p_i, q_k\rangle^2 \leq  \left(\sum_{i=1}^{r}\sum_{k=1}^d+ \sum_{k=1}^{r}\sum_{i=1}^d+ \sum_{i,k=r+1}^{d}  \right)(\lambda_i - \widetilde \lambda_k)^2 \langle p_i, q_k\rangle^2.
\]

\noindent
Since, by definition, $\widetilde \lambda_i \leq  \lambda_i,$ it follows that
\begin{align*}
\sum_{i,k=r+1}^{d} (\lambda_i - \widetilde \lambda_k)^2 \langle p_i, q_k\rangle^2 & \leq \sum_{i,k=r+1}^{d} (\lambda_i^2 + \widetilde \lambda_k^2) \langle p_i, q_k\rangle^2 \\
& \leq 2 \sum_{i=r+1}^{d}\lambda_i^2.
\end{align*}
Furthermore, we observe that 

\[
\sum_{k=1}^{r}\sum_{i=1}^d (\lambda_i - \widetilde \lambda_k)^2 \langle p_i, q_k\rangle^2 
\leq 2 \sum_{k=1}^{r}\sum_{i=1}^d (\lambda_i - \widehat \lambda_k)^2 \langle p_i, q_k\rangle^2  +2 \sum_{k=1}^{r}\sum_{i=1}^d  B(\widehat \lambda_k)^2  \langle p_i, q_k\rangle^2 ,
\]

\noindent
where, by Lemma~\ref{lem1.29}, 
\[
\sum_{k=1}^{r}\sum_{i=1}^d (\lambda_i - \widehat \lambda_k)^2 \langle p_i, q_k\rangle^2 \leq r B(\lambda_1)^2.
\]
and $B(\widehat \lambda_k)  \leq B(\widehat \lambda_1).$
We have then proved that 
\[
\sum_{k=1}^{r}\sum_{i=1}^d (\lambda_i - \widetilde \lambda_k)^2 \langle p_i, q_k\rangle^2 
\leq 2 r B(\lambda_1)^2+ 2r B(\widehat \lambda_1)^2.
\]
Applying Corollary~\ref{cor.eigen} and using the fact that $B(t + a) \leq B(t) + a/2$,
as explained after equation \eqref{eq.defB}, we deduce that 
\[
B(\widehat \lambda_1) \leq B\bigl[\lambda_1 + B(\lambda_1) \bigr] 
\leq 3 B(\lambda_1)/2.
\]
This proves that 
\[
\sum_{k=1}^{r}\sum_{i=1}^d (\lambda_i - \widetilde \lambda_k)^2 \langle p_i, q_k\rangle^2  
\leq 2 r \left[ B(\lambda_1)^2+ 9B( \lambda_1)^2/4 \right] = 13 r B( \lambda_1)^2/2.
\]
Considering that the same bound holds for
\[
\sum_{i=1}^{r}\sum_{k=1}^d(\lambda_i - \widetilde \lambda_k)^2 \langle p_i, q_k\rangle^2 ,
\]
we conclude the proof.

\newpage
\appendix

\section{Orthogonal Projectors}\label{appx}
In this appendix we introduce some results on orthogonal projectors. \\[1mm]
Let $P,\ Q: \mathbb R^d \to \mathbb R^d$ be two orthogonal projectors. 
We denote by $\mathbb S_d$ the unit sphere of $\mathbb R^d.$
By definition, 
$$\|P-Q\|_{\infty}=\sup_{x \in \mathbb S_d } \|Px-Qx\|$$
where, without loss of generality, we can take the supremum over the normalized eigenvectors of $P-Q.$ 

\vskip 2mm
\noindent
A good way to describe the geometry of $P-Q$ is to consider the 
eigenvectors of $P+Q$. 
\begin{lem}\label{lemma1} Let $x\in \mathbb S_d$ be an eigenvector of $P+Q$ with eigenvalue $\lambda.$ 
\begin{enumerate}
\item \label{eigen:1} In the case when $\lambda = 0$, then $Px = Qx = 0$, 
so that $x \in \mathbf{ker}(P) \cap \mathbf{ker}(Q)$; 
\item \label{eigen:2} in the case when $\lambda = 1$, then $PQx = QPx = 0$, 
so that 
\[
x \in \mathbf{ker}(P) \cap \mathbf{Im}(Q)  \oplus \mathbf{Im}(P) \cap \mathbf{ker}(Q); 
\]
\item \label{eigen:3} in the case when $\lambda = 2$, then $x = Px = Qx$, 
so that $x \in \mathbf{Im}(P) \cap \mathbf{Im}(Q)$; 
\item \label{eigen:4} otherwise $\lambda \in ]0,1[ \cup ]1,2[$, 
\[ 
(P-Q)^2 x = (2 - \lambda)\lambda x \neq 0, 
\] 
so that $(P-Q)x \neq 0$. Moreover 
\[ (P+Q)(P-Q) x = (2-\lambda)(P-Q)x,\]
so that $(P-Q) x$ is an eigenvector of $P+Q$ with eigenvalue $2-\lambda$.
Moreover
\[
0 < \lVert Px \rVert = \lVert Qx \rVert < \lVert 
x \rVert,
\] 
$x - Px \neq 0$, and $\bigl( Px, x - Px \bigr)$ is an orthogonal 
basis of $\mathbf{span} \bigl\{ x, (P-Q)x \bigr\}$.
\end{enumerate}
\end{lem}
\begin{proof}
The operator $P+Q$ is symmetric, positive semi-definite, and $\lVert P + Q 
\rVert \leq 2$, so that there is a basis of eigenvectors and 
all eigenvalues are in the intervall $[0,2]$. \\[1mm]
In case \ref{eigen:1}, 
$ 0 = \langle Px + Qx, x \rangle = \lVert Px \rVert^2 + \lVert Q x \rVert^2$, 
so that $Px = Qx = 0$. \\[1mm]
In case \ref{eigen:2}, $PQx = P(x - Px) = 0$ and similarly 
$QPx = Q(x - Qx) = 0$. \\[1mm]
In case \ref{eigen:3}, 
\[
\lVert Px \rVert^2 + \lVert Qx \rVert^2 = \langle (P+Q)x, x \rangle = 
2 \langle x , x \rangle = \lVert Px \rVert^2 + \lVert x - Px \rVert^2  
+ \lVert Qx \rVert^2 + \lVert x - Qx \rVert^2 , 
\]
so that $\lVert x - Px \rVert = \lVert x - Qx \rVert = 0$. \\[1mm]
In case \ref{eigen:4}, remark that 
\[
PQx = P(\lambda x - Px) = (\lambda - 1) Px
\] 
and similarly $QPx = Q ( \lambda x - Q x) = (\lambda - 1)Q x$. 
Consequently 
\[
(P-Q)(P-Q) x = (P - QP - PQ + Q) x = (2 - \lambda) (P+Q) x = (2 - \lambda)\lambda x \neq 0, 
\]
so that $(P-Q)x \neq 0$. Moreover
\[
(P+Q)(P-Q)x = (P - PQ + QP - Q)x = (2 - \lambda)(P-Q)x. 
\]
Therefore $(P-Q)x$ is an eigenvector of $P+Q$ with eigenvalue $2 - \lambda 
\neq \lambda$, so that $\langle x , (P - Q) x \rangle = 0$, since $P + Q$ 
is symmetric. As $\langle x , (P - Q) x \rangle = \lVert Px \rVert^2 
- \lVert Qx \rVert^2$, this proves that $\lVert Px \rVert = \lVert 
Qx \rVert$. Since $(P+Q)x = \lambda x \neq 0$, necessarily $\lVert Px \rVert 
= \lVert Qx \rVert > 0$. Observe now that 
\[ 
\lVert Px \rVert^2 = \frac{1}{2} \bigl( 
\lVert Px \rVert^2 + \lVert Q x \rVert^2 \bigr) 
= \frac{1}{2} \langle x , (P+Q)x \rangle = 
\frac{\lambda}{2} \lVert x \rVert^2 < \lVert x \rVert^2. 
\]     
Therefore $\lVert x - Px \rVert^2 = \lVert x \rVert^2 - \lVert P x \rVert^2
> 0$, proving that $x - Px \neq 0$. Similarly, since $P$ and $Q$ play 
symmetric roles, $\lVert Q x \rVert < \lVert x \rVert$ and 
$x - Qx \neq 0$. \\[1mm]
As $P$ is an orthogonal projector, $(Px, x-Px)$ is an orthogonal pair 
of non-zero vectors. Moreover 
\[
x = x - Px + Px \in \mathbf{span} \{ Px, x - Px \}
\]
and 
\[
(P-Q)x = 2 Px - \lambda x = (2 - \lambda) Px - \lambda(x - Px) 
\in \mathbf{span} \{ Px, x - Px \}\]
therefore, $( Px, x - Px )$ is an 
orthogonal basis of $\mathbf{span}\{x, (P-Q)x \}$.
\end{proof}
\begin{lem}
\label{lemma2}
There is an orthonormal basis $(x_i)_{i= 1}^d$ of eigenvectors 
of $P+Q$ with corresponding eigenvalues 
$\{\lambda_i, \ i=1, \dots d \}$ and indices $2 m \leq  p \leq  q \leq s$, such that 
\begin{enumerate}
\item $\lambda_{i} \in ]1,2[$, if $1 \leq i \leq m$,
\item $\lambda_{m+i} = 2 - \lambda_i$, if $1 \leq i \leq m$, 
and $x_{m+i} = \lVert (P-Q) x_{i} 
\rVert^{-1} (P-Q) x_{i}$, 
\item $x_{2m+1}, \dots, x_{ p} \in \bigl(\mathbf{Im} (P) \cap \mathbf{ker} (Q)\bigr)$, 
and $\lambda_{2m+1} = \dots = \lambda_{p} = 1$, 
\item $x_{p+1}, \dots, x_{q} \in \bigl( \mathbf{Im} (Q) \cap \mathbf{ker} (P) 
\bigr)$, and $\lambda_{p + 1} = \dots =  \lambda_{q} = 1$, 
\item $x_{q + 1}, \dots, x_{s} \in \mathbf{Im}(P) \cap \mathbf{Im}(Q)$, 
and $\lambda_{q+1} = \dots =  \lambda_{s} = 2$, 
\item $x_{s+1}, \dots, x_d \in \mathbf{ker}(P) \cap \mathbf{ker}(Q)$, 
and $\lambda_{s+1} = \dots = \lambda_d = 0$.
\end{enumerate}
\end{lem}
\begin{proof}
There exists a basis of eigenvectors of $P+Q$ (as already explained at 
the beginning of proof of Lemma ~\ref{lemma1}). From the previous lemma, 
we learn that all eigenvectors in the kernel of $P+Q$ are in 
$\mathbf{ker}(P) \cap \mathbf{ker}(Q)$, as on the other hand obviously 
$\mathbf{ker}(P) \cap \mathbf{ker}(Q) \subset \mathbf{ker}(P+Q)$ we get that 
\[
\mathbf{ker}(P+Q) = \mathbf{ker}(P) \cap \mathbf{ker}(Q).
\] 
In the same way 
the previous lemma proves that the eigenspace corresponding 
to the eigenvalue $2$ is equal to $\mathbf{Im}(P) \cap \mathbf{Im}(Q)$. 
It also proves that the eigenspace correponding to the 
eigenvalue $1$ is included in and consequently is equal 
to $\bigl( \mathbf{Im}(P) \cap \mathbf{ker}(Q) \bigr)  \oplus 
\bigl( \mathbf{ker}(P) \cap \mathbf{Im}(Q) \bigr)$, so that we can form an 
orthonormal basis of this eigenspace by taking the union 
of an orthonormal basis of $\mathbf{Im}(P) \cap \mathbf{ker}(Q)$ and 
an orthonormal basis of $\mathbf{ker}(P) \cap \mathbf{Im}(Q)$. \\[1mm]
Consider now an eigenspace corresponding to an eigenvalue 
$\lambda \in ]0,1[ \cup ]1,2[$ and let $x, y$ be two orthonormal eigenvectors  
in this eigenspace. Remark that (still from the previous lemma) 
\[ 
\langle (P-Q) x, (P-Q) y \rangle = \langle (P-Q)^2 x, y \rangle = 
(2 - \lambda) \lambda \langle x, y \rangle = 0.
\]  
Therefore, if $x_1, \dots, x_k$ is an orthonormal basis of the eigenspace $V_{\lambda}$ corresponding to the eigenvalue $\lambda$, 
then $(P-Q)x_1, \dots, (P-Q)x_k$ is an orthogonal system in $V_{2 - \lambda}$. 
If this system was not spanning $V_{2 - \lambda}$, we could add to 
it an orthogonal unit vector $y_{k+1} \in V_{2 - \lambda}$
so that $x_1, \dots, x_k, (P-Q) y_{k+1}$ would be an orthogonal 
set of non-zero vectors in $V_{\lambda}$, which would contradict 
the fact that $x_1, \dots, x_k$ was supposed to be an orthonormal basis of $V_{\lambda}$. 
Therefore,  
\[ \Bigl( \lVert (P-Q) x_i \rVert^{-1} (P-Q) x_i, \, 1 \leq i \leq k \Bigr) 
\]
is an orthonormal basis of $V_{2 - \lambda}$. Doing this construction 
for all the eigenspaces $V_{\lambda}$ such that $\lambda \in ]0,1[$ 
achieves the construction of the orthonormal basis described in 
the lemma.  
\end{proof}
\begin{lem}
Consider the orthonormal basis of the previous lemma. 
The set of vectors 
\[ 
\bigl( Px_1, \dots, Px_m, x_{2m+1}, \dots, x_p, x_{q+1}, \dots, x_s \bigr)
\] 
is an orthogonal basis of $\mathbf{Im}(P)$. 
The set of vectors 
\[ 
\bigl( Qx_1, \dots, Qx_m, x_{p+1}, \dots, x_q, x_{q+1}, \dots, x_s \bigr)
\] 
is an orthogonal basis of $\mathbf{Im}(Q)$. 
\end{lem}
\begin{proof}
According to Lemma \ref{lemma1}, $(Px_i, x_i - Px_i)$ is an orthogonal basis 
of $\mathbf{span} \{ x_i, x_{m+i} \}$, 
so that 
\[
\bigl( Px_1, \dots, Px_m, x_1 - Px_1, \dots, x_m - Px_m, x_{2m+1}, \dots, x_d
\bigr)
\]
is another orthogonal basis of $\mathbb{R}^d$. 
Each vector of this basis is either in $\mathbf{Im}(P)$ or in $\mathbf{ker}(P)$ and
more precisely
\begin{align*}
Px_1, \dots, Px_m, x_{2m+1}, \dots, x_p, x_{q+1}, \dots, x_s & \in \mathbf{Im}(P),\\ 
x_1 - Px_1, \dots, x_m - Px_m, x_{p+1}, \dots, x_q, 
x_{s+1}, \dots, x_d & \in \mathbf{ker}(P).
\end{align*}
This proves the claim of the lemma concerning $P$. Since $P$ and $Q$ 
play symmetric roles, this proves also the claim concerning $Q$, 
{\em mutatis mutandis}. 
\end{proof}
\begin{lem}
\label{lemma4}
The projectors $P$ and $Q$ have the same rank if and only if 
\[p - 2m = q - p.\] 
\end{lem} 

\begin{lem}
\label{imQ}
Assume that $\mathrm{rk}(P) = \mathrm{rk}(Q)$. Then 
\[ 
\lVert P - Q \rVert_{\infty} = \sup_{\theta \in \mathbf{Im}(Q) \cap \mathbb{S}_d}  
\lVert (P-Q) \theta \rVert.
\] 
\end{lem}
\begin{proof}
As $P-Q$ is a symmetric operator, we have
\begin{multline*}
\sup_{\theta \in \mathbb{S}_d} \lVert (P-Q) \theta \rVert^2 = 
\sup \Bigl\{ \langle (P-Q)^2 \theta, \theta \rangle \ | \ \theta \in \mathbb{S}_d 
\Bigr\} \\ = \sup \Bigl\{ \langle (P-Q)^2 \theta, \theta \rangle \ |  \
\theta \in \mathbb{S}_d \text{ is an eigenvector of } (P-Q)^2 \Bigr\}. 
\end{multline*}
Remark that the basis described in Lemma \ref{lemma2} is also 
a basis of eigenvectors of $(P-Q)^2$. More precisely, 
according to Lemma \ref{lemma1}
\begin{align*}
(P-Q)^2 x_i & = \lambda_i (2 - \lambda_i) x_i, & 1 \leq i \leq m,\\ 
(P-Q)^2 x_{m+i} & = \lambda_i ( 2 - \lambda_i) x_{m+i}, & 1 \leq i \leq m,\\
(P-Q)^2 x_i & = x_i, & 2m < i \leq q, \\ 
(P-Q)^2 x_i & = 0, & q < i \leq d.
\end{align*}
If $q - 2m > 0$, then $\lVert P - Q \rVert_{\infty} = 1$, and  $q-p > 0$,
according to Lemma \ref{lemma4}, so that 
$\lVert (P-Q)x_{p+1} \rVert = 1$, where $x_{p+1} \in \mathbf{Im}(Q)$.  
If $q = 2m$ and $m > 0$, there is $i \in \{1, \dots,  m\}$ 
such that $\lVert P - Q \rVert_{\infty}^2  = \lambda_i(2 - \lambda_i)$. 
Since $x_i$ and $x_{m+i}$ are two eigenvectors of $(P-Q)^2$ corresponding 
to this eigenvalue, all the non-zero vectors in $\mathbf{span} \{x_{i}, x_{m+i} 
\}$ (including $Qx_i$) are also eigenvectors of the same eigenspace. 
Consequently $(P-Q)^2Qx_i = \lambda_i (2 - \lambda_i) Qx_i$, 
proving that 
\[ 
\Bigl\lVert (P-Q) \frac{Q x_i}{\lVert Q x_i \rVert} \Bigr\rVert^2 
= \lambda_i (2 - \lambda_i), 
\] 
and therefore that $\sup_{\theta \in \mathbb{S}_d} \lVert (P-Q) \theta \rVert$ 
is reached on $\mathbf{Im}(Q)$. Finally, if $q = 0$, then $P-Q$ is the null 
operator, so that $\sup_{\theta \in \mathbb{S}_d} \lVert (P-Q) \theta \rVert$
is reached everywhere, including on $\mathbf{Im} (Q) \cap \mathbb{S}_d$. 
\end{proof}

\end{document}